\documentclass[reqno,11pt]{amsart}

\usepackage{amsfonts, amssymb, verbatim}
\usepackage{amssymb,epsfig,amscd, verbatim}
\usepackage{amsthm}
\usepackage{amsmath}
\usepackage[T1]{fontenc}
\usepackage{graphicx}
\usepackage{latexsym}

\newtheorem{theorem}{Theorem}
\newtheorem{lemma}[theorem]{Lemma}
\newtheorem{proposition}[theorem]{Proposition}

\theoremstyle{definition}
\newtheorem{definition}[theorem]{Definition}

\numberwithin{equation}{section} \numberwithin{theorem}{section}

\theoremstyle{remark}

\newtheorem{remarks}[theorem]{Remarks}

\def\z{\,_{\dot z}\,}

\def\cx{\rotatebox[origin=c]{180}{$\mathrm{Y}$}(x)}
\def\cy{\rotatebox[origin=c]{180}{$\mathrm{Y}$}(y)}
\def\cz{\rotatebox[origin=c]{180}{$\mathrm{Y}$}(z)}
\def\cxz{\rotatebox[origin=c]{180}{$\mathrm{Y}$}(x+z)}
\def\czy{\rotatebox[origin=c]{180}{$\mathrm{Y}$}(z+y)}

\def\cmx{\rotatebox[origin=c]{180}{$\mathrm{Y}$}(-x)}
\def\cmy{\rotatebox[origin=c]{180}{$\mathrm{Y}$}(-y)}
\def\czmy{\rotatebox[origin=c]{180}{$\mathrm{Y}$}(z-y)}
\def\D{D^{ \star}}
\def\Di{D^{\star (i)}}
\def\T{\Delta}
\def\vac{|0\rangle}                            %% vacuum vector
\def\tt{\otimes}                               %% tensor product
\def\<{\langle}
\def\>{\rangle}

\def\dsum{\displaystyle\sum}
\def\vac{\mathbf{1}}                            %% vacuum vector

\def\z{\,_{\dot z}\,}

\def\ddx{\frac{d}{dx}\,}

\def\dxyz{\delta(x-y,z)}
\def\dmyxz{\delta(-y+x,z)}
\def\dxzy{\delta(x-z,y)}

\def\dzxy{\delta(z-x,y)}

\def\v2a{(V,\z,\vac,D)}
\begin{document}

\title{Vertex Coalgebras, Co-Associator and Co-Commutator formulas}

\author[Vertex Coalgebras, Co-Associator and Co-Commutator formulas]{Florencia Orosz Hunziker$^*$}
\thanks {\textit{$^{*}$Ciem-FAMAF, Universidad Nacional de C\'ordoba, Ciudad Universitaria, (5000) C\'ordoba, Argentina. \hfill \break \indent  \, e-mail: o.flor.po@gmail.com}}
%\address{{\textit{Ciem-FAMAF, Universidad Nacional de C\'ordoba, Ciudad Universitaria, (5000) C\'ordoba, Argentina. \hfill \break \indent  e-mail: o.flor.po@gmail.com}}}

\maketitle

\date{November 28, 2012}
%\subjclass{Primary ;\\Secondary }

\begin{abstract}
Based on the definition of vertex coalgebra introduced by Hubbard \cite{H}, we prove that this notion can be reformulated using the Co-Commutator, Co-Skew symmetry and Co-Associator formulas without restrictions on the grading.
\end{abstract}
%\maketitle
%\tableofcontents

%%%%%%%%%%%%%%%%%%%%%%%%%%%%%%%%%%%%%%%%%%%%%%%%%%%%%%%%%%%%%%%%%%%%%%%%%

\section{Introduction}
The notion of vertex coalgebra was introduced by Hubbard in \cite{H} as a generalization of the notion of vertex operator coalgebra. Many of the properties described in that work are in some sense dual to the properties satisfied by vertex algebras. Since the original definition of vertex algebra was introduced by Borcherds in the 1980s, several reformulations of the definition have been studied. A vertex algebra can be defined using Lie algebra-type axioms or it can be seen as a generalization of a commutative and associative algebra with unit, focusing on the commutator or associator formulas. These formulations are introduced and thoroughly studied in \cite{LL}, \cite{K}, \cite{Li}. A vertex algebra can also be defined as a Lie conformal algebra and a left symmetric differential algebra with unit satisfying certain compatibilities, as described in \cite{BK}. These different approaches engender several equivalent definitions of vertex algebra based on different axioms. Our goal is to study whether these approaches can be dualized to obtain equivalent definitions of vertex coalgebra.

The study of these approaches is not automatic in the case of vertex coalgebras as there might be axioms in the definition of vertex algebra that do not make sense in their dual version. For instance, axioms such as "weak commutativity" and "weak associativity" do not make sense unless we require the use of grading on a vertex coalgebra (see \cite{H}). In \cite{MN}, the authors show that with a coefficient approach, the Jacobi identity can be proven to follow from the commutator and associator formulas. Based on that idea, in this preliminary version we first obtain a reformulation of the original definition of vertex coalgebra analogous to the original definition of vertex algebra introduced by Borcherds. More precisely, we prove that the original definition of vertex coalgebra can be reformulated in two equivalent definitions: One based on the Co-Commutator and the Co-Skew symmetry formulas, the other based on the Co-Associator and Co-Skew symmetry formulas. Our goal is to study whether all the possible approaches to define vertex algebras can be dualized to get equivalent definitions of vertex coalgebra in an expanded version of this work.

In the next section we introduce basic definitions and results, as well as the original definition of vertex coalgebra and its basic properties. In Section 3, following the idea presented in \cite{MN}, we prove that the coefficient version of the Co-Jacobi identity (called Co-Borcherds formula) can be deduced from weaker versions of this formula. Finally, in Section 4 we prove that the definition of vertex coalgebra can be reformulated based on the Co-Commutator, Co-Skew symmetry, Co-Associator and $\D$ formulas.

\

\vskip.2cm
%%%%%%%%%%%%%%%%%%%%%%%%%%
\section{Basic properties and the original definition}
%%%%%%%%%%%%%%%%%%%%%%%%%%

\vskip.2cm

We begin by introducing some basic definitions and results from the calculus of formal variables. A more thorough description of the concepts defined in this section, as well as the proofs of the following results can be found in \cite{FLM}, \cite{K} and \cite{LL}.

We will consider $x,y$ and $z$ commuting formal variables, and define the {\em formal $\delta$-function} to be
\begin{equation*}
\delta(x,y)= \dsum_{n \in \mathbb{Z}} x^n y^{-n-1}.
\end{equation*}
Given an integer $n$, we define
\begin{equation*}
(x\pm y)^n= \dsum_{k \geq 0} \binom{n}{k} x^{n-k} (\pm y)^k,
\end{equation*}
where $\binom{n}{k}=\frac{n(n-1)...(n-k+1)}{k!}$ for $n \in \mathbb{Z}$ and $k \geq 0$.
Note that $\delta(x-y,z)$ is a formal power series in nonnegative powers of $y$.
The formal residue `$Res_x,$' of a series $v(x)=\dsum_{n \in \mathbb{Z}} v_n x^{n}$ refers to the coefficient of the negative first power, that is
\begin{equation*}
Res_x v(x)= v_{-1}.
\end{equation*}
We will also use Taylor's Theorem, which states that given a $\mathbb{C}$-vector space $V$ and a formal series $f(x)\in V[[x,x^{-1}]],$
\begin{align} \label{Taylor}
e^{y \ddx}f(x)=f(x+y).
\end{align}
\vskip.4cm
\noindent We define the linear map \ $T: V \tt V \longrightarrow V  \tt V$ be the transposition operator defined by $T(v \tt u)=u \tt v$ for all $u,v \in  V$.

Next, we enumerate basic properties of the formal series $\delta$ that we will need throughout this work:
\begin{align}
&\mathrm{(a)} \ Res_y \, \delta(x,y) =Res_y \, \delta(y,x)=1. \label{d3}\\
&\mathrm{(b)} \ \dxyz = \delta(z+y,x). \label{d1}\\
&\mathrm{(c)} \ \dxyz - \dmyxz = \dxzy. \label{d2}
\end{align}
We will also make use of the following result (cf. Proposition 2.1.1 \cite{FLM}):

Given a formal Laurent series $X(x,y) \in (\textrm{Hom}(V,W))[[x,x^{-1}, y, y^{-1}]]$ with coefficients which are homomorphisms from a vector space $V$ to a vector space $W$, if $lim_{x \rightarrow y} X(x,y)$ exists, (i.e. when $X(x,y)$ applied to any element of $V$ setting $x=y$ leads to only finite sums in $W$)
we have:
\begin{align*}
\delta(x,y)X(x,y)=\delta(x,y)X(y,y).%\label{probar}
\end{align*}
In particular, we have
\begin{align}
Res_x \ \delta(x,y)X(x,y)=X(y,y). \label{hub1}
\end{align}

Now that we have listed the results we will need, let us recall the original definition of vertex coalgebra introduced in \cite{H}.
\begin{definition} \label{def1}
A \textit{vertex coalgebra} consists of a $\mathbb{C}$-vector space $V$, together with linear maps
\begin{align*}
\cx: &V \longrightarrow (V \otimes V)[[x,x^{-1}]], \\
c: &V\longrightarrow \mathbb{C}
\end{align*}
called the {\em coproduct} and {\em covacuum map}, respectively, satisfying the following axioms for all $v \in V$:
\vskip.2cm
\noindent $\bullet$ Left counit:
\begin{equation*} %\label{covac1}
  (c \otimes Id)\cx v=v.
\end{equation*}
\vskip.2cm
\noindent $\bullet$ Cocreation:
\begin{align} \label{covac2}
  (Id \otimes c)\cx v \in V[[x]] \ \ \ \  \textrm{ and } \ \ \ \ (Id \otimes c)\cx v|_{x=0}=v.
\end{align}
\vskip.2cm
\noindent $\bullet$ Truncation:
\begin{align}
\cx v \in (V \otimes V)((x^{-1})). \label{trun}
\end{align}
\vskip.2cm
\noindent $\bullet$ Co-Jacobi identity:
\begin{align}
&\dxzy (\cz \otimes Id)\cy= \nonumber \\
&\ \ \ \ \ \ \ \ \ \dxyz (Id \otimes \cy)\cx- \  \dmyxz (T \otimes Id)(Id \otimes \cx) \cy. \label{CoJ}
\end{align}
\vskip.2cm
\end{definition}
\begin{remarks}(1) The operator $\cx$ is linear so that, for example, $(Id \tt \cy)$ acting on the
coefficients of $\cx v \in (V \tt V )[[x, x^{-1}]]$ is well defined. \\
(2) Notice also, that when each expression is applied to any element of $V$, the coefficient of each monomial
in the formal variables is a finite sum due to the truncation condition (\ref{trun}).
\end{remarks}
As an immediate result we have,
\begin{proposition} \label{nothalf}
Let $(V, \cx ,c)$ be a vertex coalgebra as in Definition \ref{def1}. Then the following formulas hold:
\vskip.2cm
\noindent \rm{(1)} \textit{Co-Commutator formula:}
\begin{align}
Res_z \ \dxzy (\cz &\otimes Id)\cy=  \nonumber\\
&(Id \otimes \cy)\cx- (T \otimes Id)(Id \otimes \cx) \cy.  \label{cocom}
\end{align}
\vskip.1cm
\noindent \rm{(2)} \textit{Co-Associator formula:}
\begin{align}
&(\cz \otimes Id)\cy= \nonumber \\
&(Id \otimes \cy)\czy-Res_x \  \dmyxz (T \otimes Id)(Id \otimes \cx) \cy.\label{coas}
\end{align}
\end{proposition}

\vskip.2cm
\begin{proof}
 (1) Taking $Res_z$ to the Co-Jacobi identity (\ref{CoJ}) we obtain
\begin{align}
&Res_z\dxzy (\cz \otimes Id)\cy= \nonumber \\
& Res_z \ \dxyz (Id \otimes \cy)\cx-  Res_z \  \dmyxz (T \otimes Id)(Id \otimes \cx) \cy. \label{youknow}
\end{align}
Using (\ref{d3}) twice on the right side of (\ref{youknow}) we obtain (\ref{cocom}).
\vskip.2cm
\noindent (2) Taking $Res_x$ to the Co-Jacobi identity (\ref{CoJ}) we obtain
\begin{align}
&Res_x\dxzy (\cz \otimes Id)\cy= \nonumber \\
 &Res_x \ \dxyz (Id \otimes \cy)\cx-
  Res_x \  \dmyxz (T \otimes Id)(Id \otimes \cx) \cy. \label{howmuch}
\end{align}
Using (\ref{d3}) and (\ref{d1}) on the left hand side of (\ref{howmuch}) we get
\begin{align}
&(\cz \otimes Id)\cy= \nonumber \\
&Res_x \ \dxyz (Id \otimes \cy)\cx- Res_x \  \dmyxz (T \otimes Id)(Id \otimes \cx) \cy. \label{nunca}
\end{align}
Finally, using (\ref{d1}) and (\ref{hub1}) on the first term of the right hand side of (\ref{nunca}) we obtain (\ref{coas}).

\end{proof}
As pointed in \cite{H}, it is natural to question the effect of applying a formal derivative, $\ddx$, to the comultiplication operator. In order to study that effect we need to introduce a map. Given a vertex coalgebra $V$, we define the linear map $D^{\star}: V \longrightarrow V$ as
\begin{align}
D^{\star}=Res_x \ x^{-2}(Id \otimes c) \cx. \label{Ddef}
\end{align}
The following proposition obtained in \cite{H} describes the main properties of the operator $\D$ and its relationship to the formal derivative.
\begin{proposition}The map $D^{\star}$ satisfies the following properties:
\begin{align}
&\rm{(1) }
\ \ \ (D^{\star} \otimes Id) \cx = \ddx \cx. \label{ds1} \ \ \ \ \ \ \ \ \ \ \ \ \ \ \ \ \ \ \\
&\rm{(2) }\ \ \
(e^{zD^{\star}} \otimes Id) \cx= \cxz. \nonumber \\ %label{ds2}\ \ \ \ \ \ \ \ \ \ \ \ \ \ \ \ \ \ \\
&\rm{(3) }\ \ \
e^{zD^{\star}} =(Id \otimes c) \cz. \label{ds3}\ \ \ \ \ \ \ \ \ \ \ \ \ \ \ \ \ \ \\
&\rm{(4)} \ \ \ \textrm{Co-Skew symmetry:} \nonumber \\
 & \ \ \ \ \ \ \ \ \ \ \ \ \ \ \ \ \ \ \ \ \ \ \ \ \ \ \ \ \ \ \ \  T\cx= \cmx e^{xD^{\star}}.\label{coskew}\ \ \ \ \ \ \ \ \ \ \ \ \ \ \ \ \ \ \\
&\rm{(5) }  \ \ \
\ddx \cx= \cx D^{\star} - (Id \otimes D^{\star}) \cx. \nonumber %\label{ds4}\ \ \ \ \ \ \ \ \ \ \ \ \ \ \ \ \ \
\\
&\rm{(6) }\ \ \
(Id \tt e^{-zD^{\star}}) \cx e^{zD^{\star}}= \cxz. \nonumber %\label{ds5}
\end{align}
\label{hubb2}
\end{proposition}
\vskip.2cm
We consider the following expansion of the map $\cx$:

\begin{align*}
\cx v = \sum_{n \in \mathbb{Z}} \Delta_{n} (v) x^{-n-1}
\end{align*}
where $\Delta_{n}:V \longrightarrow V \otimes V $ is the coefficient of $x^{-n-1}$ in the series $\cx$.
If we look for the coefficient of $x^{-p-1}y^{-q-1}z^{-r-1}$ in the Co-Jacobi identity (\ref{CoJ}) we obtain what we will call the {\em Co-Borcherds identity}:
\begin{align}
\sum_{i \geq 0} &\binom{p}{i} (\Delta_{r+i} \otimes Id) \Delta_{p+q-i}= \nonumber \\
&\sum_{i \geq 0} (-1)^i \binom{r}{i} \bigg{[} (Id \otimes \Delta_{q+i}) \Delta_{p+r-i} - (-1)^r (T \otimes Id)(Id \otimes \Delta_{p+i}) \Delta_{q+r-i} \bigg{]}.
\end{align}
We can rewrite Definition \ref{def1} considering coefficients. With this approach we obtain the following definition of a vertex coalgebra, evidently equivalent to the one previously introduced.
\begin{definition}  A {\em vertex coalgebra} is a $\mathbb{C}$-vector space $V$ endowed with a family of linear coproducts indexed by $n \in \mathbb{Z}$ and  a linear map $c$,
\begin{align*}
\Delta_{n}: &V \longrightarrow V \otimes V  \\
c: &V\longrightarrow \mathbb{C}
\end{align*}
satisfying the following axioms:
\vskip.2cm
\noindent $\bullet$ Left counit:
\begin{align*}
  &(c \otimes Id)\Delta_{n} =Id \ \  \textrm{ if $n=-1$, \ and }  \\
  &(c \otimes Id)\Delta_{n} =0 \ \ \ \ \ \textrm{if $n \neq -1$}.
\end{align*}
\vskip.2cm
\noindent $\bullet$ Cocreation:
\begin{align*}
  &(Id \otimes c)\Delta_{n}=0   \textrm{ if $n \geq 0$, \ and }  \\
  &(Id \otimes c)\Delta_{-1} = Id.
\end{align*}
\vskip.2cm
\noindent $\bullet$ Truncation: \  For each $v \in V$ there exists $N$ such that
\begin{align*}
\Delta_{n} v=0 \textrm{ for $n \leq N$}.
\end{align*}
\vskip.2cm
\noindent $\bullet$ Co-Borcherds identity: \ For all $p, q, r \in  \mathbb{Z}$,
\begin{align}
&\sum_{i \geq 0} \binom{p}{i} (\Delta_{r+i} \otimes Id) \Delta_{p+q-i}= \nonumber \\
&\sum_{i \geq 0} (-1)^i \binom{r}{i} \bigg{[} (Id \otimes \Delta_{q+i}) \Delta_{p+r-i} - (-1)^r (T \otimes Id)(Id \otimes \Delta_{p+i}) \Delta_{q+r-i} \bigg{]}. \label{CoB}
\end{align}
\vskip.2cm \label{def2}
\end{definition}

Under this coefficient approach the definition of $D^{\star}$ corresponds to the linear map
\begin{align}
& \D :V \longrightarrow V \\
& \D = (Id \otimes c) \Delta_{-2}. \label{Cod}
\end{align}

\begin{proposition}
Let $V$ be a vertex coalgebra as in Definition \ref{def2} and $\D$ defined as in (\ref{Cod}). Then, for all $q \in \mathbb{Z},$
\vskip.3cm
\noindent \rm{(1)}
\begin{align}
(\D \tt Id) \Delta_{q}=-q\Delta_{q-1}. \label{coder}
\end{align}
\vskip.2cm
\noindent \rm{(2)}  \textit{If we denote $D^{\star(i)}= \frac{1}{i!}(D^{\star })^i$ for $i \geq 0$, then}
\begin{align}
\Di=(Id \otimes c) \Delta_{-1-i}. \label{valles}
\end{align}
\end{proposition}
\begin{proof}
(1) If we take the generating function in $x$ of (\ref{coder}) we obtain
\begin{align*}
(\D \tt Id) \cx= \ddx \cx.
\end{align*}
Therefore, (\ref{coder}) is just the coefficient formulation of (\ref{ds1}).

\noindent(2) It is easy to see that (\ref{valles}) is merely the coefficient formulation of (\ref{ds3}).

\end{proof}
Manipulating expressions that involve coefficients we will get coefficient versions of the Co-Commutator formula (\ref{cocom}), the Co-Associator formula (\ref{coas}) and the
Co-Skew symmetry formula (\ref{coskew}) based on  Definition \ref{def2}. These formulas will be axioms in the two definitions of vertex coalgebra that we will introduce later.
\begin{proposition}
Let $V$ be a vertex coalgebra as in Definition \ref{def2}. Then the following formulas hold:
\vskip.2cm
\noindent \rm{(1)}
\textit{Coefficient Co-Commutator formula:}
\begin{align}
\sum_{i \geq 0} &\binom{p}{i} (\Delta_{i} \otimes Id) \Delta_{p+q-i}=
 (Id \otimes \Delta_{q}) \Delta_{p} - (T \otimes Id)(Id \otimes \Delta_{p}) \Delta_{q}. \label{cococom}
\end{align}
\noindent \rm{(2)}
\textit{
Coefficient Co-Associator formula:}
\begin{align}
&(\Delta_{r} \otimes Id) \Delta_{q}= \nonumber \\
&\sum_{i \geq 0} (-1)^i \binom{r}{i} \bigg{[} (Id \otimes \Delta_{q+i}) \Delta_{r-i} - (-1)^r (T \otimes Id)(Id \otimes \Delta_{i}) \Delta_{q+r-i} \bigg{]}. \label{Cocoass}
\end{align}
\noindent \rm{(3)}
\textit{Coefficient Co-Skew symmetry:}
\begin{align}
T\Delta_{r} = \sum_{i \geq 0} (-1)^{r+1+i}  \Delta_{r+i} D^{\star{(i)}}.\label{cocoskew}
\end{align}
\end{proposition}

\begin{proof}
The proof follows from the fact that Definition \ref{def1} and Definition \ref{def2} are equivalent. Formulas (\ref{cococom}), (\ref{Cocoass}) and (\ref{cocoskew}) are merely the coefficient formulation of formulas (\ref{cocom}), (\ref{coas}) and (\ref{coskew}) respectively.

\end{proof}

\section{Structure of the Co-Borcherds Identity.}

Following the idea presented in \cite{MN}, we want to prove that the Co-Borcherds identity can be deduced from the Co-Commutator and the Co-Associator formulas. For that reason, we introduce a few auxiliary formulas:
\begin{align*}
&CB_1(p,q,r):=\sum_{i \geq 0} \binom{p}{i} (\Delta_{r+i} \otimes Id) \Delta_{p+q-i},\\
&CB_2(p,q,r):=\sum_{i \geq 0} (-1)^i \binom{r}{i} (Id \otimes \Delta_{q+i}) \Delta_{p+r-i},\\
&CB_3(p,q,r):=\sum_{i \geq 0} (-1)^{i+r} \binom{r}{i}(T \otimes Id)(Id \otimes \Delta_{p+i}) \Delta_{q+r-i}.
\end{align*}
Note that with the notation introduced above, the Co-Borcherds identity (\ref{CoB}) for the indices $p,q,r \in \mathbb{Z}$ corresponds to
\begin{align}
CB_1(p,q,r)=CB_2(p,q,r)-CB_3(p,q,r).\label{aa}
\end{align}
Next, we have the following formulas which are analogous to the results obtained in Section 3.2 in \cite{MN}.
\begin{proposition}\label{pqr} For all $p,q,r \in \mathbb{Z}$
\begin{align*}
CB_i(p+1,q,r)=CB_i(p,q+1,r)+CB_i(p,q,r+1),
\end{align*}
for $i=1,2,3.$
\end{proposition}
\begin{proof}
On the one hand,
\begin{align*}
CB_1(p+1,q,r)=\sum_{i \geq 0} &\binom{p+1}{i} (\Delta_{r+i} \otimes Id) \Delta_{p+1+q-i}.
\end{align*}
On the other hand,
\begin{align}
&CB_1(p,q+1,r)+CB_1(p,q,r+1) \nonumber \\
&=\sum_{i \geq 0} \binom{p}{i} (\Delta_{r+i} \tt Id) \Delta_{p+1+q-i} + \sum_{i \geq 0} \binom{p}{i}(\Delta_{r+1+i} \tt Id) \Delta_{p+q-i}
.\label{ohgod}
\end{align}
Rewriting (\ref{ohgod}) we have
\begin{align}
(\Delta_r \tt 	Id) \	\Delta_{p+1+q} + \sum_{i \geq 0} \left[\binom{p}{i+1}+ \binom{p}{i} \right] (\Delta_{r+i+1} \tt Id ) \Delta_{p+q-i}. \label{aiai}
\end{align}
Using the fact that $\binom{p}{i+1}+ \binom{p}{i}=\binom{p+1}{i+1}$ for all $i \geq 0$ we can rewrite (\ref{aiai}) as
\begin{align*}
&(\Delta_r \tt 	Id) \	\Delta_{p+1+q} + \sum_{i \geq 0}\binom{p+1}{i+1} (\Delta_{r+i+1} \tt Id ) \Delta_{p+q-i}\\
&=\sum_{i \geq 0} \binom{p+1}{i} (\Delta_{r+i} \otimes Id) \Delta_{p+1+q-i},
\end{align*}
which is exactly $CB_1(p+1,q,r).$

Analogously, we rewrite
\begin{align*}
CB_2&(p,q+1,r)+ CB_2(p,q,r+1)=
\sum_{i \geq 0} (-1)^i \binom{r}{i} (Id \otimes \Delta_{q+1+i}) \Delta_{p+r-i}\ \nonumber \\
&\ \ \ \ \ \ \ \ \ \ \ +(Id \tt \Delta_q) \Delta_{p+r+1} + \sum_{i \geq 1} (-1)^i \binom{r+1}{i} (Id \otimes \Delta_{q+i}) \Delta_{p+r+1-i}\nonumber\\
&=\sum_{i \geq 1} (-1)^{i-1}\binom{r}{i-1}(Id \tt \T_{q+i})\T_{p+r-i+1} + (Id \tt \T_q) \T_{p+r+1} \   \nonumber \\
&\ \ \ \ \ \ \ \ \ \ \  + \sum_{i \geq 1} (-1)^i \binom{r+1}{i}(Id \tt \T_{q+i})\T_{p+r+1-i} \nonumber \\
&=\sum_{i \geq 1} \bigg{[} (-1)\binom{r}{i-1}+ \binom{r+1}{i} \bigg{]} (-1)^i (Id \tt \T_{q+i}) \T_{p+r-i+1} + (Id \tt \T_q)\T_{p+r+1}  \nonumber \\
&=\sum _{i \geq 0} \binom{r}{i}(-1)^i (Id \tt \T_{q+i}) \T_{p+r-i+1}= CB_2(p+1,q,r),
\end{align*}
where we have used that $(-1)\binom{r}{i-1}+\binom{r+1}{i}=\binom{r}{i}$ for all $i \geq 1$.

Finally, we rewrite
\begin{align}
CB&_3(p,q+1,r) + CB_3(p,q,r+1)= \nonumber \\
&\sum_{i\geq 0} (-1)^{r+i} \bigg{[} \binom{r}{i}-\binom{r+1}{i}\bigg{]}(T \tt Id)(Id \tt \T_{p+i}) \T_{q+r-i+1}. \label{espalda}
\end{align}
Using that $\binom{r}{0}=1$ for all $r \in \mathbb{Z}$, and the fact that $\binom{r}{i}- \binom{r+1}{i}= -\binom{r}{i-1}$ for every $i \geq 1$, the right hand side of (\ref{espalda}) equals
\begin{align}
\sum_{i \geq 1} (-1)^{r+i+1} \binom{r}{i-1}(T \tt Id)(Id \tt \T_{p+i}) \T_{q+r-i+1}.\label{timo}
\end{align}
Rewriting (\ref{timo}) we obtain,
\begin{align*}
\sum_{i \geq 0} (-1)^{r+i} \binom{r}{i}(T \tt Id)(Id \tt \T_{p+i+1})\T_{q+r-i}=CB_3(p+1,q,r),
\end{align*}
finishing the proof.

\end{proof}
Using (\ref{aa}) and Proposition \ref{pqr}, it is easy to prove the following result:
\begin{proposition} \label{microdancing} The Co-Borcherds identity for two of the indices $(p+1,q,r), (p,q+1,r), \\
(p,q,r+1)$ imply the Co-Borcherds identity for the other index.
\end{proposition}

\noindent Now, we can prove the following statement:
\begin{theorem} \label{cesar}
The Co-Borcherds identity for all $p$ and $q$ with fixed $r$ and for all $q$ and $r$ with fixed $p$ imply the Co-Borcherds identity for all $p,q$ and $r$.
\end{theorem}
\begin{proof}
We assume the Co-Borcherds identity for all $p$ and $q$ with $r_0$ fixed and for all $q$ and $r$ with $p_0$ fixed.
Using Proposition \ref{microdancing} we have
that the Co-Borcherds identity also holds for $(p,q,r_0+1)$ for all $p,q \in \mathbb{Z}$.
Using that the Co-Borcherds identity holds for indices $(p,q+1,r_0+1)$ and $(p+1,q , r_0+1)$ and Proposition \ref{microdancing} we obtain that the Co-Borcherds identity also holds for the indices $(p,q,r_0+2)$ for all $p,q \in \mathbb{Z}$. Inductively we obtain that the Co-Borcherds identity holds for $(p,q,r_0+n)$ for all $n \in \mathbb{N}$ for every $p,q \in \mathbb{Z}$ .

Analogously, we can prove that the Co-Borcherds identity also holds for $(p_0+1,q,r) $ for all $q,r, \in \mathbb{Z}$. Inductively we obtain that the Co-Borcherds identity holds for the indices $(p_0+n,q,r)$ for all $n \in \mathbb{N}$ for every $q,r \in \mathbb{Z}$ .
Thus, it is clear that the Co-Borcherds identity holds for $(p,q,r)$ as long as $p \geq p_0$ or $r \geq r_0$.

Now, we need to analyze the case in which $p <p_0$ and $r< r_0$.
First, we note that the Co-Borcherds identity holds for indices $(p_0, s, r_0-1)$ and $(p_0 - 1, s, r_0)$ for all $s \in \mathbb{Z}$. Using Proposition \ref{microdancing} those identities imply the Co-Borcherds identity for the index $(p_0-1,s+1,r_0-1)$ for all $s \in \mathbb{Z}$. Again, using that the Co-Borcherds identity holds for indices $(p_0-1,s+1, r_0-1)$ and $(p_0-2, s+1, r_0)$ we obtain that Co-Borcherds identity holds for the index $(p_0-2, s +2, r_0-1)$. Inductively, we obtain that the Co-Borcherds identity holds for indices $(p_0-n, s+n, r_0-1)$ for all $n \geq 0$. As this identity holds for all $s \in \mathbb{Z}$ we obtain that the Co-Borcherds identity holds for indices $(p_0-n, s, r_0-1)$ for all $n \geq 0$ and $s \in \mathbb{Z}$.

Next, the fact that the Co-Borcherds identity holds for indices $(p_0-1, l, r_0 -1)$ and $(p_0 , l , r_0 -2)$ for all $l \in \mathbb{Z}$ implies that the Co-Borcherds identity holds for the index $(p_0-1,l,r_0-2)$ for all $l \in \mathbb{Z}$.
Inductively, we obtain that the Co-Borcherds identity holds for indices $(p_0-1, s ,r_0-n)$ for all $n \geq 0$ for every $s \in \mathbb{Z}$.
Thus, we have proved that if the Co-Borcherds identity holds for indices $(p,q,r)$ as long as $p \geq p_0$ or $r \geq r_0$, then it also holds for $(p,q,r)$ as long as $p \geq p_0-1$ or $r \geq r_0-1$. Continuing with this procedure, we inductively obtain that the Co-Borcherds identity holds for every indices $(p,q,r)$, finishing the proof.

\end{proof}

\section{Two equivalent definitions}
After studying the structure of the Co-Borcherds identity, we will prove that we can reformulate the definition of a vertex coalgebra focusing on either the Co-Commutator formula or the Co-Associator formula.
\begin{definition} \label{def3}
A \textit{vertex coalgebra} consists of a $\mathbb{C}$-vector space $V$, together with linear maps
\begin{align*}
\cx: &V \longrightarrow (V \otimes V)((x^{-1})) \\
c: &V\longrightarrow \mathbb{C} \\
\D: &V\longrightarrow V
\end{align*}
satisfying the following axioms:
\vskip.2cm
\noindent $\bullet$ Left counit:
\begin{equation*} %\label{covac3}
  (c \otimes Id)\cx v=v, \  \textrm{ for all $v \in V$.}
\end{equation*}
\vskip.2cm
\noindent $\bullet$ Cocreation:
\begin{align} \label{covac31}
e^{x \D}= (Id \tt c) \cx.
\end{align}
\vskip.2cm
\noindent $\bullet$ Co-Skew symmetry:
\begin{align} \label{coskew3}
T\cx=\cmx e^{x\D}.
\end{align}
\noindent $\bullet$ $\D$ formula:
\begin{align} \label{dolor}
\ddx \cx=  (\D \tt Id) \cx.
\end{align}
\vskip.2cm
\noindent $\bullet$ Co-Commutator formula:
\begin{align}
Res_z \ \dxzy (\cz &\otimes Id)\cy=  \nonumber\\
&(Id \otimes \cy)\cx- (T \otimes Id)(Id \otimes \cx) \cy . \label{cocom3}
\end{align}
\end{definition}
\vskip.2cm
\begin{remarks} \label{desd}
(a) Note that (\ref{covac31}) trivially implies (\ref{covac2}). That is, Definition \ref{def3} implies that for all $v \in V$
\begin{align*}
(Id &\otimes c)\cx v \in V[[x]] \ \ \ \ \textrm{ and } \ \ \ \ (Id \otimes c)\cx v|_{x=0}=v.
\end{align*}
\vskip.2cm
\noindent (b) Note that (\ref{covac31}) implies that the map $\D$ is the map defined in (\ref{Ddef}).
In fact, if we take the coefficient of $x^{1}$ in (\ref{covac31}), we get $\D=(Id \tt c) \T_{-2}$.
\vskip.2cm
\noindent (c) Now, using Proposition \ref{hubb2} (3), it is clear that (\ref{covac31}) is also satisfied under the conditions of Definition \ref{def1}.

\end{remarks}
Hence, $\D$ is the map we introduced earlier and the conditions regarding the map $c$ do not essentially differ from the properties described in the first definitions. We can now state one of our main results:
\vskip.5cm
\begin{theorem}\label{sozinha} Definition \ref{def3} is equivalent to Definition \ref{def1}.
\end{theorem}
\vskip.2cm
In \cite{H} (cf. Proposition 2.3), it is proved that Definition \ref{def1} implies Cocreation (\ref{ds3}), Co-Skew symmetry (\ref{coskew}), and the $\D$ formula (\ref{ds1}). In Proposition \ref{nothalf} (1) we proved that the Co-Commutator formula (\ref{cocom3}) follows from Definition \ref{def1}. Therefore, it is clear that the conditions of Definition \ref{def3}
are satisfied under the axioms of Definition \ref{def1}.
In order to prove the other implication, we first need to prove some results.
\vskip.3cm
\begin{proposition} \label{preta} The Co-Skew symmetry formula (\ref{coskew3}) and the $\D$ formula (\ref{dolor}) imply that the map $\D$ satisfies the following properties:
\vskip.2cm
\noindent \rm{(1)} $\D$\textit{-bracket formula:}
\begin{align}
&\ddx \cx= \cx D^{\star} - (Id \otimes D^{\star}) \cx \label{12}
\end{align}
\noindent \rm{(2)} \textit{Conjugation formula:}
\begin{align}
&(Id \tt e^{-zD^{\star}}) \cx e^{zD^{\star}}= \cxz \label{13}
\end{align}
\end{proposition}
\vskip.2cm
\begin{proof}
(1) Using (\ref{coskew3}), the product rule, (\ref{dolor}) and finally reapplying (\ref{coskew3}) we have
\begin{align}
\ddx \cx&= \ddx \left( T\cmx e^{x \D} \right) \nonumber \\
&= T\left(\ddx \cmx\right)e^{x \D}+ T \cmx \ddx \left( e^{x \D}\right) \nonumber \\
&= -T (\D \tt Id) \cmx e^{x \D} + T \cmx e^{x \D} \D\nonumber \\
&= -(Id \tt \D) \cx + \cx \D. \nonumber
\end{align}
\vskip.2cm
\noindent(2) Exponentiating the $\D$-bracket formula (\ref{12}) and applying Taylor's Theorem (\ref{Taylor}) we obtain (\ref{13}).

\end{proof}
Our goal is to prove that the axioms of Definition \ref{def3} imply the Co-Jacobi identity (\ref{CoJ}). We begin by proving that the Co-Associator formula (\ref{coas}) holds.
\begin{lemma}\label{under}
In the presence of Cocreation (\ref{covac31}) and the $\D$ formula (\ref{dolor}), the Co-Skew symmetry formula (\ref{coskew3}) and the Co-Commutator formula (\ref{cocom3}) imply the Co-Associator formula (\ref{coas}).
\end{lemma}
\begin{proof}
Interchanging the variables $z$ and $x$ in (\ref{cocom3}) we obtain
\begin{align}
(Id \tt \cy) \cz-(T \tt Id)(Id \tt \cz) \cy=Res_x \  \dzxy (\cx \tt Id) \cy. \label{cabeza}
\end{align}
Using that
\begin{align*}
(Id \tt T)(Id \tt T \cy) \cz = (Id \tt \cy) \cz,
\end{align*}
\begin{align*}
(Id \tt T) (\cz \tt Id) T \cy= (T \tt Id)(Id \tt \cz) \cy,
\end{align*}
and
\begin{align*}
(Id \tt T)(T \tt Id)(Id \tt \cx) T \cy= (\cx \tt Id) \cy,
\end{align*}
we have that (\ref{cabeza}) can be rewritten to obtain
\begin{align}
(Id \tt T)(Id \otimes T\cy)&\cz-(Id \tt T)(\cz \otimes Id)T \cy \nonumber \\
&=Res_x \  \dzxy (Id \tt T)(T \otimes Id)(Id \otimes \cx) T\cy. \label{cualquer}
\end{align}
We apply $(Id \tt T)$ to (\ref{cualquer}) and use that $\dzxy = \delta(y+x,z)$ to get
\begin{align}
(Id \otimes T\cy)\cz-&(\cz \otimes Id)T \cy \nonumber \\
&=Res_x \  \delta(y+x,z) (T \otimes Id)(Id \otimes \cx) T\cy. \label{sinta}
\end{align}
Next, using Co-Skew symmetry formula (\ref{coskew3}) we can rewrite (\ref{sinta}) as follows
\begin{align}
&(\cz \otimes Id)T \cy = \nonumber \\
&(Id \otimes \cmy)(Id \tt e^{y \D})\cz - Res_x \  \delta(y+x,z) (T \otimes Id)(Id \otimes \cx) T\cy. \label{por}
\end{align}
Using (\ref{13}) and (\ref{coskew3}) we have that (\ref{por}) implies
\begin{align}
&(\cz \otimes Id)\cmy e^{y \D}= \nonumber \\
&(Id \otimes \cmy)\czmy e^{y \D}-  Res_x   \ \delta(y+x,z) (T \otimes Id)(Id \otimes \cx) \cmy e^{y \D}. \label{70}
\end{align}
Multiplying by $e^{-y \D}$ both sides of (\ref{70})  we get
\begin{align}
(&\cz \otimes Id)\cmy= \nonumber \\
&(Id \otimes \cmy)\czmy-  Res_x  \ \delta(y+x,z) (T \otimes Id)(Id \otimes \cx) \cmy. \label{para}
\end{align}
Replacing $-y$ for $y$ in (\ref{para}) gives us
\begin{align}
(&\cz \otimes Id)\cy= \nonumber \\
&(Id \otimes \cy)\czy-  Res_x  \  \delta(-y+x,z) (T \otimes Id)(Id \otimes \cx) \cy, \nonumber
\end{align}
which is the Co-Associator formula.

\end{proof}

Now, we can finish the proof of Theorem \ref{sozinha}. By Remarks \ref{desd}, it remains to show that Definition \ref{def3} implies the Co-Jacobi identity (\ref{CoJ}). On the one hand, the coefficient version of the Co-Commutator formula (\ref{cocom3}) implies the Co-Borcherds identity for all $p,q \in \mathbb{Z}$ with fixed $r_0=0$. On the other hand, the Co-Associator formula, proven to hold under the axioms of Definition \ref{def3} in Lemma \ref{under}, implies the Co-Borcherds identity for all $q,r \in \mathbb{Z}$ with fixed $p_0=0$. Therefore, using Theorem \ref{cesar}, we have that the axioms of Definition \ref{def3} imply the Co-Borcherds identity for all $p,q,r \in \mathbb{Z}$. Now, it is clear that the Co-Jacobi identity holds due to the equivalence between Definition \ref{def2} and Definition \ref{def1}.

A thorough analysis of Lemma (\ref{under}) shows that we can reformulate the definition of vertex coalgebra making focus on the Co-Associator formula.

\begin{definition} \label{def4}
A \textit{vertex coalgebra} consists of a $\mathbb{C}$-vector space $V$, together with linear maps
\begin{align*}
\cx: &V \longrightarrow (V \otimes V)((x^{-1})), \\
c: &V\longrightarrow \mathbb{C} \\
\D: &V\longrightarrow V
\end{align*}
satisfying the following axioms for all $v \in V$:
\vskip.2cm
\noindent $\bullet$ Left counit:
\begin{equation*} %\label{covac4}
  (c \otimes Id)\cx v=v.
\end{equation*}
\vskip.2cm
\noindent $\bullet$ Cocreation:
\begin{align} \label{covac41}
e^{x \D}= (Id \tt c) \cx.
\end{align}
\vskip.2cm
\noindent $\bullet$ Co-Skew symmetry
\begin{align} \label{coskew4}
T\cx=\cmx e^{x\D}.
\end{align}
\noindent $\bullet$ $\D$ formula
\begin{align} \label{dolor4}
\ddx \cx=  (\D \tt Id) \cx.
\end{align}
\vskip.2cm

\noindent $\bullet$ Co-Associator formula:
\begin{align}
(\cz &\otimes Id)\cy= \nonumber \\
&(Id \otimes \cy)\czy-  Res_x \  \dmyxz (T \otimes Id)(Id \otimes \cx) \cy.\label{coass4}
\end{align}
\end{definition}
\vskip.2cm
As we did in the previous case, our goal is to use Theorem \ref{cesar} to prove that Definition \ref{def4} is equivalent to the definitions that we have already described.
\vskip.5cm
\begin{theorem} \label{younever} Definition \ref{def4} is equivalent to Definition \ref{def1}.
\end{theorem}
\vskip.2cm
Using Remarks \ref{desd}, Proposition \ref{hubb2} and (\ref{coas}) it is clear that the axioms of Definition \ref{def4} are satisfied under the conditions of Definition \ref{def1}.
In order to finish the proof of the equivalence, it is enough to show that Definition \ref{def4} implies Definition \ref{def3}. In fact, we only need to prove that the axioms of Definition \ref{def4} imply the Co-Commutator formula (\ref{cocom3}).

Using Proposition \ref{preta} it is clear that under the conditions of Definition \ref{def4} the following formulas hold:
\begin{align}
&\ddx \cx= \cx D^{\star} - (Id \otimes D^{\star}) \cx. \label{123}\\
&(Id \tt e^{-zD^{\star}}) \cx e^{zD^{\star}}= \cxz. \label{133}
\end{align}
Now, we can prove the following result:
\begin{lemma} \label{chico}
In the presence of Cocreation (\ref{covac41}) and the $\D$ formula (\ref{dolor4}), the Co-Skew symmetry formula (\ref{coskew4}) and the Co-Associator formula (\ref{coass4}) imply the Co-Commutator formula (\ref{cocom3}).
\end{lemma}
\begin{proof}
We start by multiplying by $e^{y \D}$ the Co-Associator formula (\ref{coass4}) after replacing $y$ by $-y$ to obtain
\begin{align*}
(&\cz \otimes Id)\cmy e^{y \D} \nonumber \\
&=(Id \otimes \cmy)\czmy e^{y \D}-  Res_x \  \delta(y+x,z) (T \otimes Id)(Id \otimes \cx) \cmy e^{y \D},
\end{align*}
which due to (\ref{coskew4}) and (\ref{133}) implies
\begin{align*}
&(\cz \otimes Id)T \cy  \nonumber \\
&=(Id \otimes \cmy)(Id \tt e^{y \D})\cz - Res_x \  \delta(y+x,z) (T \otimes Id)(Id \otimes \cx) T\cy \nonumber \\
&=(Id \otimes T\cy)\cz - Res_x \  \delta(y+x,z) (T \otimes Id)(Id \otimes \cx) T\cy.
\end{align*}
Thus, we have that
\begin{align}
(Id \otimes T\cy)\cz-(\cz \otimes Id)T \cy=Res_x \  \delta(y+x,z) (T \otimes Id)(Id \otimes \cx) T\cy. \label{lily}
\end{align}
Using that $\delta(y+x,z)=\dzxy$, and applying $(Id \tt T)$ to (\ref{lily}) we obtain
\begin{align}
(Id \tt T)(Id \otimes T\cy)\cz-(Id \tt T)(\cz \otimes Id)T \cy= \nonumber \\
Res_x \  \dzxy (Id \tt T)(T \otimes Id)(Id \otimes \cx) T\cy. \label{almost}
\end{align}
Using that
\begin{align*}
(Id \tt T)(Id \tt T \cy) \cz = (Id \tt \cy) \cz,
\end{align*}
\begin{align*}
(Id \tt T) (\cz \tt Id) T \cy= (T \tt Id)(Id \tt \cz) \cy,
\end{align*}
and
\begin{align*}
(Id \tt T)(T \tt Id)(Id \tt \cx) T \cy= (\cx \tt Id) \cy
\end{align*}
we have that (\ref{almost}) implies
\begin{align*}
(Id \tt \cy) \cz-(T \tt Id)(Id \tt \cz) \cy=Res_x \  \dzxy (\cx \tt Id) \cy, \label{there}
\end{align*}
which is the Co-Commutator formula.

\end{proof}

Now, using Lemma \ref{chico} we have that Definition \ref{def4} implies Definition \ref{def3}. Next, the proven equivalence between Definition \ref{def3} and Definition \ref{def1} shows that Definition \ref{def4} implies Definition \ref{def1}, finishing the proof of Theorem \ref{younever}.
\vskip.6cm
\noindent{\bf Acknowledgement}
The  author was supported  by a scholarship granted by Conicet, Consejo Nacional de Investigaciones Cient\'ificas y T\'ecnicas (Argentina). The author would like to thank her advisor, Jos\'e I. Liberati, for his guidance and help throughout this work.
\vskip.7cm
\bibliographystyle{amsalpha}

\end{document}